\RequirePackage{fix-cm}
\documentclass[smallextended]{svjour3}
\usepackage{amsmath}
\usepackage{amsfonts}
\usepackage[dvipspdf]{xcolor}

\smartqed

\def\zhou#1 {\fbox {\footnote {\ }}\ \footnotetext { From Zhou: {\color{red}#1}}}

\def\alex#1 {\fbox {\footnote {\ }}\ \footnotetext { From Alex: {\color{blue}#1}}}

\def\qi#1 {\fbox {\footnote {\ }}\ \footnotetext { From Qi: {\color{blue}#1}}}

\newcommand{\bbZ}{{\mathbb Z}}
\newcommand{\tr}{{\rm Tr}}
\newcommand{\gf}{{\mathbb F}}

\newcommand{\calD}{{\mathcal D}}
\newcommand{\dy}{{\rm DY}}
\newcommand{\rt}{{\rm RT}}
\newcommand{\p}{{\rm P}}

\begin{document}

\title{Skew Hadamard Difference Sets from Dickson Polynomials of Order $7$}

\author{Cunsheng Ding \and Alexander Pott \and Qi Wang}

\institute{
C. Ding \at Department of Computer Science and Engineering, The Hong Kong University of Science and Technology, Clear Water Bay, Kowloon, Hong Kong \\
\email{cding@ust.hk}\\
A. Pott \at Institute of Algebra and Geometry, Faculty of Mathematics, Otto-von-Guericke University Magdeburg, Universit\"atsplatz 2, 39106, Magdeburg, Germany\\
\email{alexander.pott@ovgu.de}\\
Q. Wang \at Institute of Algebra and Geometry, Faculty of Mathematics, Otto-von-Guericke University Magdeburg, Universit\"atsplatz 2, 39106, Magdeburg, Germany\\
\email{qi.wang@ovgu.de}
}

\date{Received: date / Accepted: date}

\maketitle

\begin{abstract}
Skew Hadamard difference sets are an interesting topic of study for over seventy years. For a long time, it had been conjectured the classical Paley difference sets (the set of nonzero quadratic residues in $\gf_q$ where $q \equiv 3 \bmod{4}$) were the only example in abelian groups. In 2006, the first author and Yuan disproved this conjecture by showing that the image set of $\calD_5(x^2,u)$ is a new skew Hadamard difference set in $(\gf_{3^m},+)$ with $m$ odd, where $\calD_n(x,u)$ denotes the first kind of Dickson polynomials of order $n$ and $u \in \gf_q^*$. The key observation in the proof is that $\calD_5(x^2,u)$ is a planar function from $\gf_{3^m}$ to $\gf_{3^m}$ for $m$ odd. Since then a few families of new skew Hadamard difference sets have been discovered. In this paper, we prove that for all $u \in \gf_{3^m}^*$, the set $D_u := \{\calD_7(x^2,u) : x \in \gf_{3^m}^* \}$ is a skew Hadamard difference set in $(\gf_{3^m}, +)$, where $m$ is odd and $m \not \equiv 0 \pmod{3}$. The proof is more complicated and different from that of Ding-Yuan skew Hadamard difference sets since $\calD_7(x^2,u)$ is not planar in $\gf_{3^m}$. Furthermore, we show that such skew Hadamard difference sets are inequivalent to all existing ones for $m = 5, 7$ by comparing the triple intersection numbers.


\end{abstract}

\section{Introduction}\label{sec-intro}

Let $(G,+)$ be a finite group of order $v$ (written additively). A $k$-element subset $D$ of $G$ is a $(v,k,\lambda)$ {\em difference set} if the list of differences $d_1 - d_2$ with $d_1, d_2 \in D$ and $d_1 \ne d_2$, contains each nonzero element in $G$ exactly $\lambda$ times. Equivalently, in the language of group ring, in $\bbZ[G]$, we have
$$
DD^{(-1)} = k - \lambda + \lambda G,
$$
where $D^{(-1)} = \sum_{d \in D} - d$. For two difference sets $D_1$ and $D_2$ with the same parameters in an abelian group $G$, they are said to be {\em equivalent} if there exists an automorphism $\alpha$ of $G$ and an element $b \in G$ such that $\alpha (D_1) = D_2 + b$.   
For more information on difference sets, we refer to~\cite{Bau71,BJL99,Lan83}.  

Suppose that $D$ is a difference set in $(G, +)$ of order $v$. If $G$ is the disjoint union of $D$, $-D$, and $\{0 \}$, where $-D = \{- d: d \in D\}$, then $D$ is called {\em skew} (or {\em antisymmetric}). It turns out that skew difference sets must have the parameters $(4n - 1, 2n - 1, n - 1)$ up to complementation, and are called {\em skew Hadamard difference sets}. The first family of skew Hadamard difference sets is the classical Paley difference sets in $(\gf_q, + )$, i.e., the
set of nonzero quadratic residues of $\gf_q$, where $q \equiv 3 \pmod{4}$ is a prime power. The second family of skew Hadamard difference sets had not been known until 2006~\cite{DY06}. In~\cite{DY06}, using $\calD_5(x,u)$, the first kind of Dickson polynomials of order $5$, the first author and Yuan constructed a new family of skew Hadamard difference sets in $(\gf_{3^m}, +)$. The proof relies on the fact that $\calD_5(x^2,u)$ is planar for all $u \in \gf_{3^m}^*$, i.e., $\calD_5( (x+a)^2, u) - \calD_5(x^2, u)$ permutes $\gf_{3^m}$ for each $a \in \gf_{3^m}^*$. Later, this construction was generalized by Weng, Qiu, Wang and Xiang~\cite{WQWX07} by showing that the image set of every $2$-to-$1$ planar function from $\gf_q$ to $\gf_q$ is a skew Hadamard difference set if $q \equiv 3 \pmod{4}$. The first author, Wang and Xiang also proposed another family of skew Hadamard difference sets in
$(\gf_{3^m}, +)$ using the permutation polynomials from the Ree-Tits slice symplectic spreads~\cite{DWX07}. For recent progress on skew Hadamard difference sets, for example, see~\cite{WH09,Muzy10,Feng11,FX12}. 


The idea behind the construction in~\cite{DWX07} is the following: If the image set of the classical Paley difference set under a permutation polynomial is still a difference set, then the image set must be a skew Hadamard difference set. The problem of constructing new skew Hadamard difference sets then becomes that of finding such permutation polynomials by which the difference set property of the classical Paley difference sets is preserved. In fact, the classical Paley
difference sets and the construction in~\cite{DY06} can be also viewed in such a way, but the proofs are more transparent by considering the skew Hadamard difference sets therein as the image sets of corresponding planar functions. In this paper, we make use of this idea again to construct new skew Hadamard difference sets. The permutation polynomials we apply are $\calD_7(x,u)$, the first kind of Dickson polynomials of order $7$. However, the proof is quite different from that in
the construction using the Dickson polynomials of order $5$~\cite{DY06}, noting that $\calD_7(x^2,u)$ is not a planar function in $\gf_{3^m}$. The proof techniques were once employed in~\cite{DWX07}, but are more complicated here because of the structure of $\calD_7(x,u)$. Besides a new family of skew Hadamard difference sets in $(\gf_{3^m},+)$, the contribution of this paper is also to enrich the possibilities of constructing skew Hadamard difference sets using permutation polynomials, and this may lead to a better understanding of such kinds of permutation polynomials.

The rest of the present paper is organized as follows. In Section~\ref{sec-pre}, we briefly introduce some preliminaries, and present some auxiliary results on the Dickson polynomials of order $7$. In Section~\ref{sec-main}, using the first kind of Dickson polynomials of order $7$, we construct a family of skew Hadamard difference sets. Moreover, we discuss the inequivalence relations between the new skew Hadamard difference sets and the
existing ones. In Section~\ref{sec-con}, we conclude this paper with some open problems.

\section{Preliminaries and some auxiliary results}\label{sec-pre}

In this section, we first introduce some preliminaries, which are the main tools later in our proof. We then derive some auxiliary results on the first kind of Dickson polynomials of order $7$.

\subsection{Gauss sums and the Stickelberger's theorem} 

Let $q = p^m$ be a prime power and $\xi_n$ be a fixed complex primitive $n$-th root of unity. Define $\psi: \gf_q \mapsto \mathbb{C}^*$ by $\psi (x) = \xi_p^{\tr_{q/p}(x)}$, 
which is a nontrivial character of $(\gf_q, +)$, where $\tr_{q/p}$ denotes the trace function from $\gf_q$ to $\gf_p$. Let $\chi: \gf_q^* \mapsto \mathbb{C}^*$ be a character of $\gf_q^*$. The Gauss sum is defined by
$$
g(\chi) := \sum_{x \in \gf_q^*} \chi(x) \psi(x).
$$
An elementary property of Gauss sums is the following:
$$
g(\chi) \overline{g(\chi)} = q, \quad \textrm{for each $ \chi \ne \chi_0$},
$$
where $\chi_0$ is the trivial character of $\gf_q^*$. Gauss sums can also be viewed as the Fourier coefficients in the Fourier expansion of $\psi$ in terms of the characters of $\gf_q^*$, i.e., for each $x \in \gf_q^*$,
\begin{equation}\label{eqn-invf}
\psi(x) = \frac{1}{q-1} \sum_{\chi \in X} g(\chi) \chi^{-1}(x),
\end{equation}
where $X$ denotes the character group of $\gf_q^*$. In order to utilize Gauss sums, we need Stickelberger's theorem on the prime ideal factorization of Gauss sums. 

For each prime ideal $\mathfrak{p}$ in $\bbZ[\xi_{q-1}]$ lying over $p$, we have that $\bbZ[\xi_{q-1}]/\mathfrak{p}$ is a finite field of order $q$, denoted by $\gf_q$. Let $\omega_{\mathfrak p}$ be the Teichm{\"u}ller character on $\gf_q$, which is an isomorphism
$$
\omega_{\mathfrak p} : \gf_q^* \mapsto \{1, \xi_{q-1}, \xi_{q-1}^2, \ldots, \xi_{q-1}^{q-2} \}
$$
satisfying $\omega_{\mathfrak p} (x) \pmod{ {\mathfrak p}} = x$, for all $x \in \gf_q^*$. It is clear that $\omega_{\mathfrak p}$ generates the character group of $\gf_q^*$.

For each $\bbZ[\xi_{q-1}]$, the prime ideal ${\mathfrak p}$ splits into a prime power ideal in $\bbZ[\xi_{q-1}, \xi_p]$, and ${\mathfrak p} = {\mathfrak B}^{p-1}$ with ${\mathfrak B}$ a prime ideal in $\bbZ[\xi_{q-1}, \xi_p]$. Let $v_{\mathfrak B}$ be the ${\mathfrak B}$-adic valuations in $\bbZ[\xi_{q-1}, \xi_p]$. The Stickelberger's theorem is the following~\cite[p. 344]{BERW98}.

\begin{theorem}\label{thm-sti}
  Let $q = p^m$ be a prime power. For each integer $a = \sum_{i=0}^{m-1} a_i p^i$ with $0 < a < q-1$, define the {\em $p$-adic weight} of $a$ as $w(a) = \sum_{i=0}^{m-1} a_i$. Then
  $$
  v_{\mathfrak B}(g(\omega_{\mathfrak p}^{-a})) = w(a).
  $$
\end{theorem}

\subsection{The Dickson polynomials and some auxiliary results}

Let $m > 0$ be an integer and $q = p^m$ with $p$ a prime. For each $u \in \gf_q$, the {\em Dickson polynomial} $\calD_n(x,u)$ of the first kind is defined by~\cite{LMT93,LN97}
$$
\calD_n(x,u) := \sum_{j=0}^{\lfloor n/2 \rfloor} \frac{n}{n-j} {n-j \choose j} (-u)^j x^{n-2j} .
$$
In~\cite{DY06}, the first author and Yuan showed that the polynomial $g_u (x) = \calD_5(x^2, -u)$ is a planar function from $\gf_{3^m}$ to $\gf_{3^m}$, and then constructed skew Hadamard difference sets by using the image set of $g_u (x)$. We also refer to~\cite{WQWX07} for an explanation of this construction from the viewpoint of presemifields. When the order $n = 7$, the Dickson polynomial is
$$
\calD_7(x,u) = x^7  - u x^5 - u^2 x^3 - u^3 x .
$$

Hereafter, let $q = 3^m$. We will employ the Dickson polynomial $\calD_7(x,u)$ to construct new skew Hadamard difference sets in $(\gf_q, +)$. To this end, we need the following auxiliary results.

\begin{lemma}\cite{LMT93,LN97}\label{lem-dpper}
The Dickson polynomial $\calD_n(x,u)$, for each $u \in \gf_q^*$, is a permutation polynomial of $\gf_q$ if and only if $\gcd(n, q^2 - 1) = 1$.
\end{lemma}

\begin{corollary}\label{coro-dp7per}
  The Dickson polynomial $\calD_7(x,u)$, for $u \in \gf_q^*$, is a permutation polynomial of $\gf_q$ if and only if $m \not \equiv 0 \pmod{3}$.
\end{corollary}

\begin{proof}
  Note that $\gcd(7,q^2-1) = 1$ if and only if $m \not\equiv 0 \pmod{3}$. The desired conclusion then follows from Lemma~\ref{lem-dpper}. 
\qed
\end{proof}

For each $u \in \gf_q^*$, define
\begin{equation}\label{eqn-du}
  D_u := \{ \calD_7(x^2, u) : x \in \gf_q^* \}.
\end{equation}
For the image set $D_u$, we have the following result.

\begin{lemma}\label{lem-is}
  If $m$ is odd and $m \not \equiv 0 \pmod{3}$, for each $u \in \gf_q^*$, we have
$$
D_u \cap (-D_u) = \emptyset,
$$
and
$$
D_u \cup (-D_u) \cup \{0\} = \gf_q.
$$
\end{lemma}

\begin{proof}
First, suppose that $\calD_7(x^2, u) = - \calD_7(y^2, u)$ for some $x, y \in \gf_q^*$, which means $\calD_7(x^2, u) = \calD_7(-y^2, u)$. Note that the Dickson polynomial $\calD_7(x,u)$ is a permutation polynomial when $m \not \equiv 0 \pmod{3}$ by Lemma~\ref{coro-dp7per}. We then have $x^2 = - y^2$, which implies that $-1$ is a square in $\gf_q$. Since $q = 3^m$ and $m$ is odd, this is a contradiction. Thus, we have $D_u \cap (-D_u) = \emptyset$. 

Second, by the first assertion of this lemma, it suffices to prove that $|D_u| = (q-1)/2$. For each $u \in \gf_q^*$, note that $\calD_7(0,u) = 0 $. Since $\calD_7(x,u)$ is a permutation polynomial, we have $\calD_7(x^2,u) = 0$ if and only if $x = 0$. It then follows that $|D_u| = (3^m - 1)/2$, which completes the proof.
\qed
\end{proof}

We remark that if we replace the Dickson polynomial $\calD_7(x,u)$ in (\ref{eqn-du}) with any other permutation polynomial of $\gf_q$, we have the similar result to Lemma~\ref{lem-is}.

\section{A family of skew Hadamard difference sets}\label{sec-main} 

\subsection{The construction}

In this section, we show that the image set defined by (\ref{eqn-du}) is a skew Hadamard difference set in $(\gf_q, +)$. We first present the main theorem as follows.

\begin{theorem}\label{thm-main}
  Let $u \in \gf_q^*$ and $D_u$ be defined as in {\rm (\ref{eqn-du})}. If $m$ is odd  and $m \not \equiv 0 \pmod{3}$, $D_u$ is a skew Hadamard difference set in $(\gf_q, +)$.
\end{theorem}

Since we have already proved that $D_u$ is skew in Lemma~\ref{lem-is}, it suffices to prove that $D_u$ is a difference set. In order to do this, we need to prove that for every nontrivial additive character $\psi$ of $\gf_q$, 
$$
\psi(D_u) \overline{\psi(D_u)} = \frac{q+1}{4} .
$$
We will use the following lemma~\cite{CXS94} to simplify the proof.

\begin{lemma}\label{lem-sim}
  Let $G$ be an abelian group of order $p^m$, where $p$ is a prime congruent to $3$ modulo $4$, and $m$ is an odd integer. Let $D$ be a subset of $G$ such that in $\bbZ[G]$, 
  $$
  D + D^{(-1)} = G - 1,
  $$
  and $D^{(t)} = D$ for every nonzero quadratic residue $t$ modulo $p$. If for every nontrivial additive character $\psi$ of $G$, 
  $$
  \psi (D) \equiv \frac{p^{(m-1)/2} - 1}{2} \pmod{p^{(m-1)/2}},
  $$
  then $D$ is a difference set in $G$.
\end{lemma}

Now we are ready to show that $D_u$ is a skew Hadamard difference set in $(\gf_q, +)$.

\begin{proof}[Theorem~\ref{thm-main}]
  
  Since $1\in \bbZ / 3\bbZ$ is the only nonzero quadratic residue modulo $3$, the condition of Lemma~\ref{lem-sim} can certainly be satisfied, i.e., $D_u^{(t)} = D_u$ for every nonzero quadratic residue $t$ modulo $3$. By Lemma~\ref{lem-is}, $D_u$ is skew. Then by Lemma~\ref{lem-sim}, it suffices to show that for every nontrivial additive character $\psi_\beta: \gf_q \mapsto {\mathbb C}^*$, 
\begin{equation}\label{eqn-goal}
  \psi_\beta(D_u) \equiv \frac{3^{(m-1)/2} - 1}{2} \pmod{3^{(m-1)/2}},
\end{equation}
where $\psi_\beta = \xi_3^{\tr(\beta x)}$, $\beta \in \gf_q^*$, $\xi_3 = e^{2\pi i/3}$, and $\tr$ denotes the trace function from $\gf_q$ to $\gf_3$.

We now compute the left hand side of (\ref{eqn-goal}). Let $\chi$ be the multiplicative quadratic character of $\gf_q$. Then
\begin{eqnarray*}
  \psi_\beta(D_u) & = & \sum_{x \in \gf_q^*} \psi_\beta (\calD_7(x,u)) \frac{\chi(x) + 1}{2} \\ 
  & = &  \frac{1}{2} \left( \sum_{x \in \gf_q^*} \psi_\beta (\calD_7(x,u)) \chi(x) + \sum_{x \in \gf_q^*} \psi_\beta (\calD_7(x,u)) \right) \\
  & = & \frac{1}{2} \left( \sum_{x \in \gf_q^*} \psi_\beta(\calD_7(x,u)) \chi(x) - 1 \right),
\end{eqnarray*}
where in the last equality we used the facts that $\calD_7(x,u)$ is a permutation polynomial of $\gf_q$ and $\calD_7(0,u) = 0$. It then follows that (\ref{eqn-goal}) is equivalent to
\begin{equation}\label{eqn-goal2}
  \sum_{x \in \gf_q^*} \psi_\beta (\calD_7(x,u)) \chi(x) \equiv 0 \pmod{3^{(m-1)/2}}.
\end{equation}
Let $S_{\beta} = \sum_{x \in \gf_q^*} \psi_\beta(\calD_7(x,u)) \chi(x)$. We then have
\begin{eqnarray*}
  S_\beta & = & \sum_{x \in \gf_q^*} \xi_3^{\tr(\beta^3 x^{21} - \beta^3 u^3 x^{15} - (\beta u^2 + \beta^3 u^9) x^3)} \chi(x) \\
  & = & \sum_{y \in \gf_q^*} \xi_3^{\tr(\beta^3 y^7 - \beta^3 u^3 y^5 - (\beta u^2 + \beta^3 u^9) y)} \chi(y) \\
  & = & \pm \sum_{z \in \gf_q^*} \xi_3^{\tr(z^7 - \beta^{3 + 7^{-1} \cdot 5 \cdot (-3)} u^3 z^5 - (\beta^{1 + 7^{-1} \cdot (-3)} u^2 + \beta^{3 + 7^{-1} \cdot (-3)} u^9 ) z ) } \chi (z),
\end{eqnarray*}
where $7^{-1}$ denotes the multiplicative inverse of $7$ modulo $q-1$.

Let $\gamma_u = - (\beta^{1 + 7^{-1} \cdot (-3)} u^2 + \beta^{3 + 7^{-1} \cdot (-3)} u^9 )$, and $\eta_u = - \beta^{3 + 7^{-1} \cdot 5 \cdot (-3)} u^3$. Clearly, $\eta_u \ne 0$. 

If $\gamma_u = 0$,  the equation (\ref{eqn-goal2}) becomes 
\begin{equation}\label{eqn-goal31}
  \sum_{z \in \gf_q^*} \xi_3^{\tr(z^7 + \eta_u z^5)} \chi(z)  \equiv 0 \pmod{3^{(m-1)/2}} .
\end{equation}

Using Fourier inversion (\ref{eqn-invf}), we have for each $z \in \gf_q^*$, $$
\xi_3^{\tr(z)} = \frac{1}{q-1} \sum_{b=0}^{q-2} g(\omega^{-b}) \omega^b(z),
$$
where $\omega$ is the Teichm{\"u}ller character on $\gf_q$. Then we get
\begin{eqnarray*}
  \lefteqn{\sum_{z \in \gf_q^*} \xi_3^{\tr( z^7 + \eta_u z^5)} \chi(z)} \\
  & = & \sum_{z \in \gf_q^*} \xi_3^{\tr(\eta_u z^5)} \chi(z) \frac{1}{q-1} \sum_{b=0}^{q-2} g(\omega^{-b}) \omega^b(z^7) \\
  & = & \sum_{z \in \gf_q^*} \xi_3^{\tr(\eta_u z^5)} \omega^{- \frac{q-1}{2} } (z)  \frac{1}{q-1} \sum_{b=0}^{q-2} g(\omega^{-b}) \omega^{7b}(z) \\
  & = & \frac{1}{q-1} \sum_{b=0}^{q-2} g(\omega^{-b}) \sum_{z \in \gf_q^*} \xi_3^{\tr(\eta_u z^5)} \omega^{- \frac{q-1}{2} + 7b} (z) \\
  & = & \frac{1}{q-1} \sum_{b=0}^{q-2} g(\omega^{-b}) \sum_{z^5 \in \gf_q^*} \xi_3^{\tr(\eta_u z^5)} \omega^{- \frac{q-1}{2} + 5^{-1} \cdot 7b} (z^5),
\end{eqnarray*}
where $5^{-1}$ is the multiplicative inverse of $5$ modulo $q-1$. Let $z = \eta_u^{- 5^{-1}} y$, then we have
\begin{eqnarray*}
  \lefteqn{\sum_{z \in \gf_q^*} \xi_3^{\tr( z^7 + \eta_u z^5)} \chi(z)} \\
  & = & \frac{1}{q-1} \sum_{b=0}^{q-2} g(\omega^{-b}) \sum_{\eta_u^{-1} y^5 \in \gf_q^*} \xi_3^{\tr(y^5)} \omega^{- \frac{q-1}{2} + 5^{-1} \cdot 7b} (y^5) \omega^{- \frac{q-1}{2} + 5^{-1} \cdot 7b} (\eta_u^{-1} )\\
  & = & \frac{1}{q-1} \sum_{b=0}^{q-2} g(\omega^{-b}) g(\omega^{- \frac{q-1}{2} + 5^{-1} \cdot 7b}) \omega^{- \frac{q-1}{2} + 5^{-1} \cdot 7b} (\eta_u^{-1}).
\end{eqnarray*}
Thus,
$$
S_\beta = \pm \frac{1}{q-1} \sum_{b=0}^{q-2} g(\omega^{-b}) g(\omega^{- \frac{q-1}{2} + 5^{-1} \cdot 7b}) \omega^{- \frac{q-1}{2} + 5^{-1} \cdot 7b} (\eta_u^{-1}).
$$

Fix each prime ideal ${\mathfrak p}$ in $\bbZ[\xi_{q-1}]$ lying over $3$ and let ${\mathfrak B}$ be the prime ideal of $\bbZ[\xi_{q-1}, \xi_3]$ lying over ${\mathfrak p}$. Since $v_{\mathfrak B} (3) = 2$, we have 
\begin{equation}\label{eqn-ob}
S_\beta \equiv 0 \pmod{3^{(m-1)/2}}  \iff  v_{\mathfrak B}(S_\beta) \ge m-1.
\end{equation}
It then follows that
\begin{eqnarray*}
  \lefteqn{S_\beta \equiv 0 \pmod{3^{(m-1)/2}}}\\
  & \iff &  v_{\mathfrak B} \left( \sum_{b=0}^{q-2} g(\omega^{-b}) g(\omega^{- \frac{q-1}{2} + 5^{-1} \cdot 7b}) \omega^{- \frac{q-1}{2} + 5^{-1} \cdot 7b} (\eta_u^{-1}) \right) \ge m - 1.
\end{eqnarray*}
By Theorem~\ref{thm-sti}, we have for each $b$ with $0 \le b \le q - 2$,
$$
v_{\mathfrak B} (g(\omega^{-b}) g(\omega^{- \frac{q-1}{2} + 5^{-1} \cdot 7b})) = w(b) + w \left( \frac{q-1}{2} - 5^{-1} \cdot 7b \right).
$$
Thus, to prove (\ref{eqn-goal31}), it suffices to prove that for each $b$ with $0 \le b \le q-2$,
\begin{equation}\label{eqn-goal41}
  w(b) + w \left( \frac{q-1}{2} - 5^{-1} \cdot 7b \right) \ge m - 1.
\end{equation}
We will prove the statement in Theorem~\ref{thm-goal41}. Since the proof of (\ref{eqn-goal41}) is lengthy, we put it in Appendix~\ref{sec-app}.

If $\gamma_u \ne 0$, we need to prove that
\begin{equation}\label{eqn-goal32}
  \sum_{z \in \gf_q^*} \xi_3^{\tr(z^7 + \eta_u z^5 + \gamma_u z)} \chi(z) \equiv 0 \pmod{3^{(m-1)/2}}.
\end{equation}

Similarly, using Fourier inversion (\ref{eqn-invf}), we have 
\begin{eqnarray*}
  \lefteqn{\sum_{z \in \gf_q^*} \xi_3^{\tr(z^7 + \eta_u z^5 + \gamma_u z)} \chi (z)} \\
  & = & \sum_{z \in \gf_q^*} \xi_3^{\tr(\eta_u z^5 + \gamma_u z)} \omega^{- \frac{q-1}{2}} (z) \frac{1}{q-1} \sum_{b_1 = 0}^{q-2} g(\omega^{-b_1}) \omega^{b_1} (z^7) \\
  & = & \frac{1}{q-1} \sum_{b_1 = 0}^{q-2} g(\omega^{- b_1}) \sum_{z \in \gf_q^*} \xi_3^{\tr(\eta_u z^5 + \gamma_u z)} \omega^{- \frac{q-1}{2} + 7 b_1} (z) \\
  & = & \frac{1}{q-1} \sum_{b_1 = 0}^{q-2} g(\omega^{-b_1}) \sum_{y \in \gf_q^*} \xi_3^{\tr(y^5 + \gamma_u \eta_u^{- 5^{-1}} y)} \omega^{- \frac{q-1}{2} + 7b_1}(\eta_u^{-5^{-1}}y) \\
  & = & \frac{1}{q-1} \sum_{b_1 = 0}^{q-2} g(\omega^{-b_1}) \sum_{y \in \gf_q^*} \xi_3^{\tr(\gamma_u \eta_u^{-5^{-1}} y)} \omega^{- \frac{q-1}{2} + 7 b_1} (\eta_u^{-5^{-1}} y ) \\
  & & \qquad \cdot \frac{1}{q-1} \sum_{b_2 = 0}^{q-2} g(\omega^{- b_2}) \omega^{b_2}(y^5) \\
  & = & \frac{1}{(q-1)^2} \sum_{b_1 = 0}^{q-2} g(\omega^{-b_1}) \sum_{b_2 = 0}^{q-2} g(\omega^{-b_2}) \omega^{- \frac{q-1}{2} + 7 b_1} (\eta_u^{-5^{-1}}) \\
  & & \qquad  \cdot \sum_{y \in \gf_q^*} \xi_3^{\tr(\gamma_u \eta_u^{- 5^{-1}} y)} \omega^{- \frac{q-1}{2} + 7 b_1 + 5 b_2} (y) \\
  & = & \frac{1}{(q-1)^2} \sum_{b_1 = 0}^{q-2} \sum_{b_2 = 0}^{q-2} g(\omega^{-b_1}) g(\omega^{-b_2}) g(\omega^{- \frac{q-1}{2} + 7 b_1 + 5 b_2}) \\
  & & \qquad \cdot \omega^{- \frac{q-1}{2} + 7 b_1}( \eta_u^{- 5^{-1}}) \omega^{- \frac{q-1}{2} + 7 b_1 + 5 b_2 } ( \gamma_u^{-1} \eta_u^{5^{-1}}) .
\end{eqnarray*}

By (\ref{eqn-ob}), we have
\begin{eqnarray*}
  \lefteqn{S_\beta \equiv 0 \pmod{3^{(m-1)/2}}} \\
  & \iff & v_{\mathfrak B} \bigg( \sum_{b_1 = 0}^{q-2} \sum_{b_2 = 0}^{q-2} g(\omega^{-b_1}) g(\omega^{-b_2}) g(\omega^{- \frac{q-1}{2} + 7 b_1 + 5 b_2}) \\
  & & \qquad \cdot \quad \omega^{- \frac{q-1}{2} + 7 b_1}( \eta_u^{- 5^{-1}}) \omega^{- \frac{q-1}{2} + 7 b_1 + 5 b_2 } ( \gamma_u^{-1} \eta_u^{5^{-1}}) \bigg) \ge m - 1.
\end{eqnarray*}
From Theorem~\ref{thm-sti}, it then follows that for all $b_1$, $b_2$, with $0 \le b_1 \le q-2$ and $0 \le b_2 \le q-2$, 
\begin{eqnarray*}
  \lefteqn{v_{\mathfrak B} \left( \sum_{b_1 = 0}^{q-2} \sum_{b_2 = 0}^{q-2} g(\omega^{-b_1}) g(\omega^{-b_2}) g(\omega^{- \frac{q-1}{2} + 7 b_1 + 5 b_2})\right)} \\
  & = &  w(b_1) + w(b_2) + w \left(\frac{q-1}{2} - 7b_1 - 5b_2\right).
\end{eqnarray*}
Therefore, in order to prove (\ref{eqn-goal32}), we need to show that for all $b_1$, $b_2$, with $0 \le b_1 \le q-2$ and $0 \le b_2 \le q-2$, 
\begin{equation}\label{eqn-goal42}
  w(b_1) + w(b_2) + w \left( \frac{q-1}{2} - 7 b_1 - 5b_2 \right) \ge m- 1.
\end{equation}
We will prove (\ref{eqn-goal42}) in Theorem~\ref{thm-goal42}, which is also put in Appendix~\ref{sec-app}.

Combining Theorem~\ref{thm-goal41}, Theorem~\ref{thm-goal42}, and Lemma~\ref{lem-is}, the conclusion follows.
\qed
\end{proof}

\subsection{Inequivalence of skew Hadamard difference sets}\label{sec-equi}

In this section, we discuss the equivalence relations of skew Hadamard difference sets. In~\cite{DWX07}, all the known families of skew Hadamard difference sets in $(\gf_{3^m},+)$ with $m$ odd were summarized and the inequivalence relations were verified by computer for $m = 5, 7$. For the skew Hadamard difference sets constructed from the first kind of Dickson polynomials of order $7$, we first have the following result, similar to both the Ding-Yuan skew Hadamard difference sets in~\cite{DY06} and the skew
Hadamard difference sets from Ree-Tits permutation polynomials in~\cite{DWX07}. 

\begin{theorem}\label{thm-dp7}
  Let $u \in \gf_q^*$. The skew Hadamard difference sets $D_u$ in $(\gf_q, +)$ constructed in Theorem~\ref{thm-main} are equivalent to either of the following:
  \begin{itemize}
    \item[(i)] the difference set $D_1 = \{ \calD_7(x^2, 1) | x \in \gf_q^* \}$;
    \item[(ii)] the difference set $D_{-1}= \{ \calD_7(x^2, -1) | x \in \gf_q^* \}$,
  \end{itemize}
  where $\calD_7(x,u)$ denotes the first kind Dickson polynomial of order $7$.
\end{theorem}

\begin{proof}
Note that $\calD_7(x,u) = x^7 - u x^5 - u^2 x^3 - u^3 x$. For each $b \in \gf_q^*$, we have 
\begin{equation}\label{eqn-dpequi}
b^7 \calD_7(x,u) = b^7 x^7 - ub^7 x^5 - u^2 b^7 x^3 - u^3 b^7 x = \calD_7(bx, ub^2).
\end{equation}

Setting $u = 1$ in (\ref{eqn-dpequi}), we have 
$$
b^7 \calD_7(x^2, 1) = \calD_7(bx^2, b^2).
$$
Thus, it is easily seen that $D_{b^2} = b^7 D_1$ if $b$ is a square in $\gf_q^*$, while $D_{b^2} = - b^7 D_1$ if $b$ is a nonsquare. Thus, for every square $u \in \gf_q^*$, $D_u$ is equivalent to $D_1$.

With similar argument, we can also prove that for every nonsquare $u \in \gf_q^*$, $D_u$ is equivalent to $D_{-1}$.
\qed
\end{proof}

In general, it is difficult to distinguish inequivalent skew Hadamard difference sets~\cite{DY06,DWX07}. In~\cite{DY06,DWX07}, two different techniques were used to verified the inequivalence relations by computer, respectively. Here we use the method in~\cite{Bau71,DWX07} to determine the inequivalence relations for $m= 5, 7$. In the sequel, we use $\p$ to denote the classical Paley difference set, $\dy(1)$ and $\dy(-1)$ to denote the two classes of skew Hadamard difference sets
in~\cite{DY06}, $\rt(1)$ and $\rt(-1)$ to denote the two classes of skew Hadamard difference sets in~\cite{DWX07}, respectively. 

For a difference set $D$ in $(\gf_q, +)$, define
$$
{\rm T} \{a,b\} := | D \cap (D+ a) \cap (D+ b)|,
$$
where $a, b \in \gf_q^*$. These numbers ${\rm T} \{a,b\}$ are called the {\em triple intersection numbers}. It is clear that two equivalent difference sets have the identical distribution of the triple intersection numbers. 

When $m=5$, the distributions of the triple intersection numbers of these difference sets in $(\gf_{3^m}, +)$ are summarized in the following table. Note that the exponents denote the multiplicities of the corresponding triple intersection numbers.
\begin{table}[h]
  \begin{center}
\begin{tabular}{lc} \hline
   DS & Triple intersection numbers with multiplicities  \\
   \hline
  $\p$ & $26^{1815} 27^{3630} 28^{1815} 29^{7260} \cdots 33^{1815}$ \\
  $\dy(1)$ &  $23^{15} 24^{30} 25^{285} 26^{1245} \cdots 35^{45}$ \\
  $\dy(-1)$ & $24^{75} 25^{435} 26^{1155} 27^{2385} \cdots 35^{120}$  \\
  $\rt(1)$ & $24^{75} 25^{330} 26^{1155} 27^{2535} \cdots 35^{105}$  \\
  $\rt(-1)$ & $24^{90} 25^{330} 26^{1095} 27^{2655} \cdots 35^{120}$ \\
  $D_1$ & $23^{30} 24^{60} 25^{390} 26^{1110} \cdots 36^{45}$  \\
  $D_{-1}$ & $23^{15} 24^{75} 25^{330} 26^{1005} \cdots 36^{15}$ \\
  \hline
\end{tabular}
\end{center}
\end{table}

It is easily checked that the distributions of all these $7$ classes of skew Hadamard difference sets are pairwise distinct. Therefore, we conclude that all these $7$ classes of skew Hadamard difference sets are pairwise inequivalent for $m = 5$.

When $m = 7$, we only need to check the maximum and the minimum triple intersection numbers of these difference sets in $(\gf_{3^m}, +)$, as listed in the following table.
\begin{table}[h]
  \begin{center}
\begin{tabular}{lcc} \hline
   DS & MIN & MAX  \\
   \hline
  $\p$ & $261$ & $284$ \\
  $\dy(1)$ & $246$ & $300$ \\
  $\dy(-1)$ & $248$ & $297$  \\
  $\rt(1)$ & $250$  & $295$ \\
  $\rt(-1)$ & $249$ & $296$ \\
  $D_1$ & $244$ & $301$  \\
  $D_{-1}$ & $246$ & $299$ \\
  \hline
\end{tabular}
\end{center}
\end{table}

Thus, for $m = 7$, all these $7$ classes of skew Hadamard difference sets are also pairwise inequivalent. Consequently, we make the following conjecture.

\begin{conjecture}\label{con-equi}
  For all odd $m > 7$ with $m \not\equiv 0 \pmod{3}$, the seven skew Hadamard difference sets $\p$, $\dy(1)$, $\dy(-1)$, $\rt(1)$, $\rt(-1)$, $D_1$ and $D_{-1}$ are pairwise inequivalent. 
\end{conjecture}

\section{Conclusions}\label{sec-con}

Using the first kind of the Dickson polynomials $\calD_7(x,u)$, we showed that for all $u \in \gf_q^*$, $D_u$, the image set of $\calD_7(x^2, u)$ is a skew Hadamard difference set in $(\gf_{3^m}, +)$, where $m$ is odd and $m \not \equiv 0 \pmod{3}$. Furthermore we proved that every such skew Hadamard difference set is equivalent to either $D_1$ or $D_{-1}$. By comparing the triple intersection numbers, with the help of computer, we verified that both of the two skew Hadamard difference sets in
$(\gf_{3^m},+)$ are inequivalent to all existing ones with the same parameters for $m = 5, 7$. To conclude the paper, we would like to make the following remarks. 
\begin{itemize}
  \item[(1)] Based on the numerical results (up to $m = 19$), it seems impossible to obtain new skew Hadamard difference sets using the Dickson polynomials of higher order.
  \item[(2)] We emphasize that this is the second family of skew Hadamard difference sets using permutation polynomials (the first one in~\cite{DWX07}). Both of the two examples still look mysterious, and it is more interesting to characterize such permutation polynomials that can be used to construct skew Hadamard difference sets. On the other hand, because of the direct relation between skew Hadamard difference sets and even planar functions~\cite{WQWX07}, one may ask whether the skew Hadamard difference sets constructed in this paper are in fact the image sets of certain new planar functions. 
  \item[(3)] Since when $m$ is even and $m \not\equiv 0 \bmod{3}$, the first kind of Dickson polynomials of order $7$ are still permutation polynomials, it is of interest to consider the image set of $\calD_7(x^2,u)$ in $(\gf_{3^m},+)$ where $m$ is even. With the help of magma~\cite{BCP97}, we verified that the image set of $\calD_7(x^2,u)$ is a Paley type partial difference set in $(\gf_{3^m}, +)$ for $m = 4, 8$. However, the proof is still missing.
  \item[(4)] We also mention that Muzychuk~\cite{Muzy10} gave a powerful construction of skew Hadamard difference sets in elementary abelian groups of order $q^3$ with the prime power $q \equiv 3 \bmod{4}$, in comparison that the skew Hadamard difference sets constructed here are in $(\gf_{3^m}, +)$ with $m$ odd.  
\end{itemize}

\begin{acknowledgements}
Cunsheng Ding's research is supported by the Hong Kong Research Grants Council under project no. 601311. Qi Wang's research is supported by the Alexander von Humboldt (AvH) Stiftung/Foundation.
\end{acknowledgements}

\appendix
\section{Appendix}\label{sec-app}

In this appendix, we prove both (\ref{eqn-goal41}) and (\ref{eqn-goal42}). Let $q = 3^m$ with $m$ odd and $m \not \equiv 0 \pmod{3}$. In both two proofs, we will always use the following theorem~\cite{HHKWX09}.

\begin{theorem}{\rm \cite[Theorem 4.1]{HHKWX09}}\label{thm-carry}
  Let $a^{(1)}, a^{(2)}, \ldots, a^{(n)}$ be $n$ integers, and let the integer $s$ satisfy
  $$
  s \equiv l_1 a^{(1)} + l_2 a^{(2)} + \cdots + l_n a^{(n)} \pmod{p^m - 1},
  $$
  for some nonzero integers $l_1, l_2, \ldots, l_n$. Suppose that $s$ and $a^{(1)}, a^{(2)}, \ldots, a^{(n)}$ have $p$-ary representations $s = \sum_{i=1}^{m-1} s_i p^i$ and $a^{(j)} = \sum_{i=0}^{m-1} a_i^{(j)} p^i$ for $j = 1, 2, \ldots, n$, where the $p$-ary digits $s_i$ and all $a_i^{(j)}$'s are integers in $\{0,1, \ldots, p-1\}$. Then there exists a unique integer sequence $c = c_{-1}, c_0, \ldots, c_{m-1}$ with $c_{-1} = c_{m-1}$ such that
  $$
  p c_i + s_i = c_{i-1} + \sum_{j=1}^n l_j a_i^{(j)} \quad (0 \leq i \leq m-1).
  $$
  Moreover, if we define
  $$
  l_+ = \sum_{\begin{array}{c}j=1 \\ l_j>0 \end{array}}^n l_j, \qquad l_- = \sum_{\begin{array}{c}j=1 \\ l_j < 0 \end{array}}^n l_j,
  $$
  then $l_- - 1 \leq c_i \leq l_+$, and furthermore
  $$
  l_- \leq c_i \leq l_+ - 1,
  $$
  for $i = 0, 1, \ldots, m-1$ provided that $a^{(j)} \not \equiv 0 \bmod{p^m - 1}$ for some $j = 1, \ldots, n$.
\end{theorem}

\begin{theorem}\label{thm-goal41}
For each $a$, $0 \le a \le q-2$, we have
\begin{equation}\label{eqn-dsum1}
  w(a) + w \left( \frac{q-1}{2} - 5^{-1} \cdot 7 a \right) \ge m,
\end{equation}
where $w(a)$ is the digit sum of $a$ defined in Theorem~\ref{thm-sti}.
\end{theorem}

\begin{proof}
  Since $m$ is odd, we have $\gcd(q-1, 5) = 1$. Thus, it is equivalent to prove
  \begin{equation}\label{eqn-dsum12}
    w(5a) + w \left( \frac{q-1}{2} - 7 a \right) \ge m.
  \end{equation} 

  For each $0 \leq a \leq q-2$, suppose that $a$ has the ternary representation $a = \sum_{i=0}^{m-1} a_i 3^i$ with $a_i \in \{0, 1, 2 \}$. We extend $a_0, a_1, \ldots, a_{m-1}$ to a periodic ternary sequence with period $m$, i.e., $a_i = a_j$ whenever $i \equiv j \pmod{m}$. We then have
  \begin{eqnarray*}
    \frac{q-1}{2} - 7a & = & \frac{q-1}{2} - 3^2 a + 3 a - a \\
    & = & \sum_{i=0}^{m-1} ( 1 - a_{i-2} + a_{i-1} - a_i ) 3^i \pmod{3^m - 1} \\
    & = & \sum_{i=0}^{m-1} ( 1 + (2 - a_{i-2}) + a_{i-1} + (2 - a_i)) 3^i \pmod{3^m - 1} \\
    & = & \sum_{i=0}^{m-1} ( 5 + a_{i-1} - a_{i-2} - a_i) 3^i \pmod{3^m - 1}.
  \end{eqnarray*}
For each $i$, let
\begin{equation}\label{eqn-b1}
b_i = 5 + a_{i-1} - a_{i-2} - a_i,
\end{equation}
and by Theorem~\ref{thm-carry}, there exists a unique sequence $\{c_i\}$ such that 
\begin{equation}\label{eqn-s1}
s_i = b_i - 3 c_i + c_{i-1},
\end{equation}
with $s_i \in \{0, 1, 2\}$, where $b_i \in \{1, 2, \ldots, 7\}$ and $c_i \in \{0, 1, 2, 3\}$ is the carry from the $i$th digit to the $(i+1)$th digit in the modular summation of $\frac{q-1}{2}$, $-3^2 a$, $3a$ and $-a$. Similarly, for $5a$, we have 
$$
5a = \sum_{i=0}^{m-1} ( 4 + a_{i-2}  - a_{i-1} - a_i ) 3^i \pmod{3^m - 1}.
$$
For each $i$, let
\begin{equation}\label{eqn-d1}
  d_i = 4 + a_{i-2} - a_{i-1} - a_i,
\end{equation}
and there exists a unique sequence $\{e_i\}$ such that 
\begin{equation}\label{eqn-t1}
  t_i = d_i - 3 e_i + e_{i-1},
\end{equation}
with $t_i \in \{0, 1, 2\}$, where $d_i \in \{0 , 1, \ldots, 6\}$ and $e_i \in \{0, 1, 2\}$ is the carry from the $i$th digit to the $(i+1)$th digit in the modular summation of $3^2 a$, $-3a$ and $- a$. It then follows that
\begin{eqnarray*}
  \lefteqn{w(5a) + w\left( \frac{q-1}{2} - 7a\right)} \\
  & = & \sum_{i=0}^{m-1}s_i + \sum_{i=0}^{m-1} t_i \\
  & = & \sum_{i=0}^{m-1} ( 5 + a_{i-1} - a_{i-2} - a_i - 3 c_i + c_{i-1} ) \\
  & & \quad + \sum_{i=0}^{m-1} (4 + a_{i-2} - a_{i-1} - a_i - 3 e_i + e_{i- 1}) \\
  & = & 9m - 2 \sum_{i=0}^{m-1} a_i - 2\sum_{i=0}^{m-1} c_i - 2\sum_{i=0}^{m-1} e_i.
\end{eqnarray*}
Thus, to prove Theorem~\ref{thm-goal41}, it suffices to prove
$$
\sum_{i=0}^{m-1} (a_i + c_i + e_i ) \leq 4m.
$$

We now prove the inequality above by discussing all the possible values of $a_i + c_i + e_i$, where $a_i \in \{0, 1, 2\}$, $c_i \in \{0, 1, 2, 3\}$ and $e_i \in \{0, 1, 2\}$. 
\begin{itemize}
  \item[I.] $a_i + c_i + e_i = 7$. 
    
    There is only one possibility: $a_i = 2$, $c_i = 3$ and $e_i = 2$. On one hand, by (\ref{eqn-s1}), $s_i = b_i - 3 c_i + c_{i-1} \ge 0 $ implies that $ b_i \geq 6$. On the other hand, by (\ref{eqn-b1}), we have
    $$
    b_i = 5 + a_{i-1} - a_{i-2} - a_i = 3 + a_{i-1} - a_{i-2} \leq 5,
    $$
    contradicting to $b_i \geq 6$.

  \item[II.] $a_i + c_i + e_i = 6$. 
    
    There are three possibilities: (i) $a_i = 2$, $c_i = 3$ and $e_i = 1$; (ii) $a_i = 2$, $c_i = 2$ and $e_i = 2$; (iii) $a_i = 1$, $c_i = 3$ and $e_i = 2$.   
    \begin{itemize}
    \item[II.(i).] $a_i = 2$, $c_i = 3$ and $e_i = 1$. 
      
      This is impossible according to Case I.
    
    \item[II.(ii).] $a_i = 2$, $c_i = 2$ and $e_i = 2$. 
      
      By (\ref{eqn-t1}), $t_i = d_i - 3 e_i + e_{i-1} \geq 0 $ implies that 
      \begin{equation}\label{eqn-t22}
	4 \leq d_i \leq 6.
      \end{equation}
      Meanwhile, by (\ref{eqn-d1}), we have
      \begin{equation}\label{eqn-d22}
	d_i = 4 + a_{i-2} - a_{i-1} - a_i = 2 + a_{i-2} - a_{i-1} \leq 4.
      \end{equation}
      It then follows from (\ref{eqn-t22}) and (\ref{eqn-d22}) that $d_i = 4$, $a_{i-2} = 2$ and $a_{i-1} = 0$. Then by (\ref{eqn-b1}), we have
      $$
      b_i = 5 + a_{i-1} - a_{i-2} - a_i = 1.
      $$
      However, by (\ref{eqn-s1}), $s_i = b_i - 3 c_i + c_{i-1} \geq 0$ implies that $b_i \geq 3$, and this leads to a contradiction.

    \item[II.(iii).] $a_i = 1$, $c_i = 3$ and $e_i = 2$. 
      
      By (\ref{eqn-s1}), note that $s_i = b_i - 3 c_i + c_{i-1} \geq 0 $. We then have
      \begin{equation}\label{eqn-s23}
	6 \leq b_i \leq 7.
      \end{equation}
      By (\ref{eqn-b1}), we also have
      \begin{equation}\label{eqn-b23}
	b_i = 5 + a_{i-1} - a_{i-2} - a_i = 4 + a_{i-1} - a_{i-2} \leq 6.
      \end{equation}
      Thus, (\ref{eqn-s23}) and (\ref{eqn-b23}) imply that $b_i = 6$, $a_{i-1} = 2$ and $a_{i-2} = 0$.  By (\ref{eqn-d1}), we can determine the value of $d_i$ as
      $$
      d_i = 4 + a_{i-2} - a_{i-1} - a_i = 1.
      $$
      Note that $t_i = d_i - 3 e_i + e_{i-1} \geq 0 $. Then we have $ 4  \leq d_i \leq 6$, this is a contradiction to $d_i = 1$.
  \end{itemize}

  \item[III.] $a_i + c_i + e_i = 5$. 
    
    There are six possibilities: (i) $a_i = 0$, $c_i = 3$ and $e_i = 2$; (ii) $a_i = 1$, $c_i = 2$ and $e_i = 2$; (iii) $a_i = 1$, $c_i = 3$ and $e_i = 1$; (iv) $a_i = 2$, $c_i = 1$ and $e_i = 2$; (v) $a_i = 2$, $c_i = 2$ and $e_i = 1$; (vi) $a_i = 2$, $c_i = 3$ and $e_i = 0$. 
    \begin{itemize}
      \item[III.(i).] $a_i = 0$, $c_i = 3$ and $e_i = 2$. 
	
	By (\ref{eqn-s1}) and (\ref{eqn-t1}), we have
	$6 \leq b_i \leq 7$ and $4 \leq d_i \leq 6$. Adding up these two inequalities, we have 
	\begin{equation}\label{eqn-bd31}
	  10 \leq b_i + d_i \leq 13.
	\end{equation}
	By (\ref{eqn-b1}) and (\ref{eqn-d1}), we have 
	\begin{eqnarray*}
	  b_i & = & 5 + a_{i-1} - a_{i-2} - a_i = 5 + a_{i-1} - a_{i-2} \\
	  d_i & = & 4 + a_{i-2} - a_{i-1} - a_i = 4 + a_{i-2} - a_{i-1}.
	\end{eqnarray*}
	Adding up the two inequalities above, we get
	$$
	b_i + d_i = 9,
	$$
	contradicting to (\ref{eqn-bd31}).

      \item[III.(ii).] $a_i = 1$, $c_i = 2$ and $e_i = 2$. 
	
	It follows from (\ref{eqn-b1}) and (\ref{eqn-d1}) that 
	$$
	b_i + d_i = 9 - 2 a_i = 7.
	$$
	Since $b_i \ge 3 c_i - c_{i-1} \ge 3$ and $d_i \ge 3 e_i - e_{i-1} \ge 4$, we have $b_i = 3$ and $d_i = 4$. Then again by (\ref{eqn-b1}) and (\ref{eqn-d1}), we get 
	\begin{eqnarray*}
	  c_{i-1} & \ge & 3 c_i - b_i = 3, \\
	  e_{i-1} & \ge & 3 e_i - d_i = 2.
	\end{eqnarray*}
	Thus, $c_{i-1} = 3$ and $e_{i-1} = 2$, which is impossible according to Case III.(i).
      
      \item[III.(iii).] $a_i = 1$, $c_i = 3$ and $e_i = 1$. 
	
	By (\ref{eqn-s1}) and (\ref{eqn-b1}), we get 
	$$
	b_i \ge 3 c_i - c_{i-1} \ge 6
	$$
	and 
	$$
	b_i = 5 + a_{i-1} - a_{i-2} - a_i = 4 + a_{i-1} - a_{i-2} \le 6,
	$$
	respectively. Thus, we have $b_i = 6$, $a_{i-1} = 2$ and $a_{i-2} = 0$. Again by (\ref{eqn-s1}), $c_{i-1} \geq 3 c_i - b_i = 3$. With similar argument by (\ref{eqn-t1}) and (\ref{eqn-d1}), we have $d_i = 1$. It then follows that $e_{i-1} \geq 3 e_i - d_i = 2$. This is also impossible by Case III.(i).
      
      \item[III.(iv).] $a_i = 2$, $c_i = 1$ and $e_i = 2$. 
	
	By (\ref{eqn-t1}) and (\ref{eqn-d1}), we have 
	$$
	d_i \ge 3 e_i - e_{i-1} \ge 4
	$$
	and 
	$$
	d_i = 4 + a_{i-2} - a_{i-1} - a_i = 2 + a_{i-2} - a_{i-1} \le 4,
	$$
	respectively. The two inequalities can be satisfied only if $d_i = 4$, $a_{i-2} = 2$ and $a_{i-1} = 0$. Again by (\ref{eqn-t1}), we have $e_{i-1} \ge 3 e_i - d_i = 2$ and then $e_{i-1} = 2$. 
	From (\ref{eqn-t1}) and (\ref{eqn-d1}), it follows that
	$$
	d_{i-1} \ge 3 e_{i-1} - e_{i-2} \ge 4
	$$
	and 
	$$
	d_{i-1} = 4 + a_{i-3} - a_{i-2} - a_{i-1} = 2 + a_{i-3} \le 4.
	$$
	Thus, we have $d_{i-1} = 4$ and $a_{i-3} = 2$. Using (\ref{eqn-t1}) again, we get $e_{i-2} \ge 3 e_{i-1} - d_{i-1} = 2$ and then $e_{i-2} = 2$. Applying (\ref{eqn-t1}) and (\ref{eqn-d1}) again, we have
	$$
	d_{i-2} \geq 3 e_{i-2} - e_{i-3} \ge 4
	$$
	and 
	$$
	d_{i-2} = 4 + a_{i-4} - a_{i-3} - a_{i-2} = a_{i-4} \le 2,
	$$
	respectively. This is a contradiction. 
      
      \item[III.(v).] $a_i = 2$, $c_i = 2$ and $e_i = 1$. 
	
	By (\ref{eqn-b1}) and (\ref{eqn-d1}), we have 
	$$
	b_i + d_i = 9 - 2a_i = 5,
	$$
	while by (\ref{eqn-s1}) and (\ref{eqn-t1}), we get 
	$$
	b_i \geq 3 c_i - c_{i-1} \geq 3
	$$
	and 
	$$
	d_i \geq 3 e_i - e_{i-1} \geq 1,
	$$
	respectively. Thus, there are two subcases: one is $b_i=3$ and $d_i = 2$; the other is $b_i = 4$ and $d_i = 1$. In the first subcase, $a_{i-1} - a_{i-2} = 0$. With (\ref{eqn-b1}) and (\ref{eqn-d1}), $c_{i-1} \ge 3 c_i - b_i = 3$ and $e_{i-1} \geq 3 e_i - d_i = 1$. We summarize the information in the following table.
	\begin{equation*}
	  I := \left[ \begin{array}{ccc}
	    a_{i-1} = a_{i-2} & c_{i-1} = 3 & e_{i-1} \geq 1 
	  \end{array} \right]
	\end{equation*}
	

By the previous argument in Case III.(iii), only the case that $a_{i-1} = a_{i-2} = 0$, $c_{i-1}= 3$ and $e_{i-1} = 1$ is possible. However, by (\ref{eqn-s1}) and (\ref{eqn-b1}), we have
$$
b_{i-1} \ge 3 c_{i-1} - c_{i-2} \ge 6
$$
and
$$
b_{i-1} = 5 + a_{i-2} - a_{i-3} - a_{i-1} = 5 - a_{i-3} \le 5.
$$
This is a contradiction. Now we look at the second subcase, i.e., $b_i = 4$, $d_i = 1$ and $a_{i-1} - a_{i-2} = 1$. Note that by (\ref{eqn-s1}) and (\ref{eqn-t1}), we have $c_{i-1} \ge 3 c_i - b_i = 2$ and $e_{i-1} \ge 3 e_i - d_i = 2$. To sum up, we list the information in the following table.
	\begin{equation*}
	  I := \left[ \begin{array}{ccc}
	    a_{i-1} - a_{i-2} =1 & c_{i-1} \ge 2 & e_{i-1} = 2 
	  \end{array} \right]
	\end{equation*}


This is also impossible by Case III.(ii). 

      \item[III.(vi).] $a_i = 2$, $c_i = 3$ and $e_i = 0$. 
	
	This is impossible by Case I.
    \end{itemize}
\end{itemize}

Thus, for each $i$, we have $a_i + c_i + e_i \leq 4$, which completes the proof.
\qed
\end{proof}

\begin{theorem}\label{thm-goal42}
  For each $a$ with $0 \le a \le q-2$ and $b$ with $0 \le b \le q-2$, we have
\begin{equation}\label{eqn-dsum2}
  w(a) + w(b) + w \left( \frac{q-1}{2} - 7 a - 5 b \right) \ge m,
\end{equation}
where $w(a)$ is the digit sum of $a$ defined in Theorem~\ref{thm-sti}.
\end{theorem}

  For each $0 \leq a \leq q-2$, and $0 \leq b \leq q-2$, suppose that $a$ and $b$ have the ternary representation $a = \sum_{i=0}^{m-1} a_i 3^i$ with $a_i \in \{0, 1, 2 \}$ and $b = \sum_{i=0}^{m-1} b_i 3^i$ with $b_i \in \{ 0, 1, 2\}$, respectively. Here we view both $\{a_i\}$ and $\{b_i\}$ as ternary sequences with period $m$. It then follows that
  \begin{eqnarray*}
    \lefteqn{\frac{q-1}{2} - 5a - 7b}\\
    & = & \frac{q-1}{2} - 3a - 3a + a - 3^2 b + 3b - b \\
    & = & \sum_{i=0}^{m-1} ( 1 - a_{i-1} - a_{i-1} + a_i - b_{i-2} + b_{i-1} - b_i) 3^i \pmod{3^m - 1} \\
    & = & \sum_{i=0}^{m-1} ( 1 + (2 - a_{i-1}) + (2 - a_{i-1}) + a_i \\
    & & \quad + (2 - b_{i-2}) + b_{i-1} + (2 - b_i)) 3^i \pmod{3^m - 1} \\
    & = & \sum_{i=0}^{m-1} ( 9 - 2 a_{i-1} + a_i -  b_{i-2} + b_{i-1} - b_i) 3^i \pmod{3^m - 1}.
  \end{eqnarray*}
For each $i$, let 
\begin{equation}\label{eqn-d2}
  d_i = 9 - 2 a_{i-1} + a_i - b_{i-2} + b_{i-1} - b_i.
\end{equation}
It is easily seen that $d_i \in \{ 1, 2, \ldots, 13 \}$. By Theorem~\ref{thm-carry}, there exists a unique sequence $\{c_i\}$ such that
\begin{equation}\label{eqn-s2}
  s_i = d_i - 3c_i + c_{i-1},
\end{equation}
with $s_i \in \{0, 1, 2\}$, where $c_i \in \{0, 1, \ldots, 6\}$ is the carry from the $i$th digit to the $(i+1)$th digit in the modular summation of $\frac{q-1}{2}$, $-3a$, $-3a$, $a$, $-3^2 b$, $3b$ and $- b$. We then have
\begin{eqnarray*}
  \lefteqn{w(a) + w(b) + w\left( \frac{q-1}{2} - 5a - 7b \right)} \\
  & = & \sum_{i=0}^{m-1} a_i + \sum_{i=0}^{m-1} b_i + \sum_{i=0}^{m-1} s_i \\
  & = & \sum_{i=0}^{m-1} a_i + \sum_{i=0}^{m-1} b_i \\
  & & \quad + \sum_{i=0}^{m-1} (9 - 2 a_{i-1} + a_i - b_{i-2} + b_{i-1} - b_i - 3c_i + c_{i-1}) \\
  & = & 9m - 2\sum_{i=0}^{m-1} c_i .
\end{eqnarray*}
Thus, it suffices to prove 
\begin{equation}\label{eqn-ci0}
\sum_{i=0}^{m-1} c_i \leq 4m.
\end{equation}
To prove Theorem~\ref{thm-goal42}, we need the following results on the sequence $\{c_i\}$.

\begin{lemma}\label{lem-ci1}
For each $i$, $c_i \leq 5$.
\end{lemma}

\begin{proof}
  Suppose that for some $i$, $c_i = 6$. Note that $s_i = d_i - 3 c_i + c_{i-1} \geq 0$, then $d_i \geq 3 c_i - c_{i-1} \ge 12$ and $c_{i-1} \geq 3 c_i - d_i \ge 5$. This implies $c_{i-1} = 6$ or $c_{i-1} = 5$. If $c_{i-1} = 6$, since $s_{i-1} = d_{i-1} - 3c_{i-1} + c_{i-2} \geq 0$, we have $12 \leq d_{i-1} \leq 13$. Thus, adding up the two inequalities by (\ref{eqn-d2}), we get
  $$
  24 \leq d_i + d_{i-1} = 18 - 2 a_{i-2} - a_{i-1} + a_i - b_{i-3} - b_i,
  $$
  but this inequality cannot be satisfied for all $\{a_i\}$. If $c_{i-1} = 5$, similar argument also leads to a contradiction.
\qed
\end{proof}

\begin{lemma}\label{lem-ci2}
  If for some $i$, $c_i = 5$, then $c_{i-1} \leq 4$ and $c_{i+1} \leq 4$.
\end{lemma}

\begin{proof}
  It suffices to prove $c_{i-1} \leq 4$ only. Suppose to the contrary that $c_{i-1} = 5$. Note that $s_i = d_i - 3 c_i + c_{i-1} \ge 0 $ and $s_{i-1} = d_{i-1} - 3 c_{i-1} + c_{i-2} \ge 0$. Then we have $d_i \geq 3 c_i - c_{i-1} \ge 10$ and $d_{i-1} \ge 10$. By (\ref{eqn-d2}), we have 
  \begin{equation}\label{eqn-ci21}
  10 \leq d_i = 9 - 2 a_{i-1} + a_i - b_{i-2} + b_{i-1} - b_i
\end{equation}
  and 
  \begin{equation}\label{eqn-ci22}
  10 \leq d_{i-1} = 9 - 2 a_{i-2} + a_{i-1} - b_{i-3} + b_{i-2} - b_{i-1}.
\end{equation}
  Adding up these two inequalities, we get
  $$
  20 \leq d_i + d_{i-1} = 18 - 2a_{i-2} - a_{i-1} + a_i - b_{i-3} - b_i.
  $$
  This can be satisfied only if $a_i = 2$ and $a_{i-1} = a_{i-2} = b_i =  b_{i-3} = 0 $. Back to (\ref{eqn-ci21}) and (\ref{eqn-ci22}), we have $b_{i-2} - b_{i-1} \le 1$ and $b_{i-2} - b_{i-1} \ge 1$, respectively. Thus, $b_{i-2} - b_{i-1} = 1$ and $d_i = d_{i-1} = 10$. Again by (\ref{eqn-s2}), we have $c_{i-2} \geq 3 c_{i-1} - d_{i-1} = 5$, which implies that $c_{i-2} = 5$ by Lemma~\ref{lem-ci1}. Then $d_{i-2} \geq 3 c_{i-2} - c_{i-3} \ge 10$. On the other hand, 
  \begin{eqnarray*}
  d_{i-2} & = & 9 - 2a_{i-3} + a_{i-2} - b_{i-4} + b_{i-3} - b_{i-2} \\
  & = & 9 - 2 a_{i-3} - b_{i-4} - b_{i-2} \\
  & \leq & 9.
\end{eqnarray*}
This is a contradiction.
\qed
\end{proof}

\begin{lemma}\label{lem-ci3}
  If for some $i$, $c_i = 5$ and $c_{i-1} = 4$, then $c_{i-2} \leq 4$.
\end{lemma}

\begin{proof}
  Suppose to the contrary that $c_{i-2} = 5$. Since $s_i = d_i - 3 c_i + c_{i-1} \geq 0 $, we have
  \begin{eqnarray*}
    d_i & \geq & 3 c_i - c_{i-1} = 11,\\
    d_{i-1} & \geq & 3 c_{i-1} - c_{i-2} = 7,\\
    d_{i-2} & \geq & 3 c_{i-2} - c_{i-3} \geq 11,
  \end{eqnarray*}
  where the last inequality follows from Lemma~\ref{lem-ci2}. By (\ref{eqn-d2}), we have
  \begin{eqnarray}
    11 & \leq & d_i = 9 - 2 a_{i-1} + a_i - b_{i-2} + b_{i-1} - b_i, \label{eqn-ci31} \\
    7 & \leq & d_{i-1} = 9 - 2a_{i-2} + a_{i-1} - b_{i-3} + b_{i-2} - b_{i-1}, \label{eqn-ci32} \\
    11 & \leq & d_{i-2} = 9 - 2 a_{i-3} + a_{i-2} - b_{i-4} + b_{i-3} - b_{i-2}. \label{eqn-ci33}
  \end{eqnarray}
Adding up the three inequalities, we get 
$$
29 \leq 27 -  2 a_{i-3} - a_{i-2} - a_{i-1} + a_i - b_{i-4} - b_{i-2} - b_i.
$$
This can be satisfied only if $a_i = 2$ and $a_{i-1} = a_{i-2} = a_{i-3} = b_i = b_{i-2} = b_{i-4} = 0$. From (\ref{eqn-ci33}), it follows that $ 11 \leq d_{i-2} = 9 + b_{i-3} $. Thus, $b_{i-3} = 2$ and $d_{i-2} = 11$. Again by (\ref{eqn-s2}), we have $c_{i-3} \geq 3 c_{i-2} - d_{i-2} = 4$. By Lemma~\ref{lem-ci2}, $c_{i-3} = 4$. Since $s_{i-3} = d_{i-3} - 3 c_{i-3} + c_{i-4} \geq 0$, we have $d_{i-3} \geq 3 c_{i-3} - c_{i-4} \geq 7$. Then by (\ref{eqn-d2}), 
\begin{eqnarray*}
  7 \leq d_{i-3} & = & 9 - 2 a_{i-4} + a_{i-3} - b_{i-5} + b_{i-4} - b_{i-3} \\
  & = & 7 - 2 a_{i-4} - b_{i-5} \\
  & \leq & 7.
\end{eqnarray*}
It then follows that $d_{i-3} = 7$ and $a_{i-4} = b_{i-5} = 0$. Thus, $c_{i-4} \geq 3 c_{i-3} - d_{i-3} = 5$, and then $c_{i-4} = 5$ by Lemma~\ref{lem-ci1}. On one hand, $d_{i-4} \geq 3 c_{i-4} - c_{i-5} \geq 11$. On the other hand, 
\begin{eqnarray*}
d_{i-4} & = & 9 - 2a_{i-5} + a_{i-4} - b_{i-6} + b_{i-5} - b_{i-4} \\
& = & 9 - 2a_{i-5} - b_{i-6} \\
& \leq & 9,
\end{eqnarray*}
which leads to a contradiction.
\qed
\end{proof}

\begin{lemma}\label{lem-ci4}
  If for some $i$, $c_i = 5$ and $c_{i-1} = c_{i-2} = 4$, then $c_{i-3} \leq 4$.
\end{lemma}

\begin{proof}
  Suppose to contrary that $c_{i-3} = 5$. Note that $s_i = d_i - 3c_i + c_{i-1} \geq 0$ for each $i$. Then
  \begin{eqnarray*}
    d_i & \geq & 3 c_i - c_{i-1} = 11, \\
    d_{i-1} & \geq & 3 c_{i-1} - c_{i-2} = 8, \\
    d_{i-2} & \geq & 3 c_{i-2} - c_{i-3} = 7, \\
    d_{i-3} & \geq & 3 c_{i-3} - c_{i-4} \geq 11,
  \end{eqnarray*}
  where in the last inequality we used the fact that $c_{i-4} \le 4$ by Lemma~\ref{lem-ci2}. Using (\ref{eqn-d2}), we have
  \begin{eqnarray}
    11 & \leq & d_i = 9 - 2 a_{i-1} + a_i - b_{i-2} + b_{i-1} - b_i, \label{eqn-ci41} \\
    8 & \leq & d_{i-1} = 9 - 2 a_{i-2} + a_{i-1} - b_{i-3} + b_{i-2} - b_{i-1}, \label{eqn-ci42} \\
    7 & \leq & d_{i-2} = 9 - 2 a_{i-3} + a_{i-2} -  b_{i-4} + b_{i-3} - b_{i-2}, \label{eqn-ci43} \\
    11 & \leq & d_{i-3} = 9 - 2 a_{i-4} + a_{i-3} - b_{i-5} + b_{i-4} - b_{i-3}. \label{eqn-ci44}
  \end{eqnarray}
  Adding up (\ref{eqn-ci41}) and (\ref{eqn-ci42}), (\ref{eqn-ci43}) and (\ref{eqn-ci44}), respectively, we get
  \begin{eqnarray}
    19 & \leq & 18 - 2a_{i-2} - a_{i-1} + a_i - b_{i-3} - b_i, \label{eqn-ci45} \\
    18 & \leq & 18 - 2a_{i-4} - a_{i-3} + a_{i-2} - b_{i-5} - b_{i-2}. \label{eqn-ci46}
  \end{eqnarray}
  By (\ref{eqn-ci45}), we know that $a_{i-2} = 0$. Then by (\ref{eqn-ci46}), we have $a_{i-3} = a_{i-4} = b_{i-2} = b_{i-5} = 0$. The inequality (\ref{eqn-ci44}) becomes
  $$
  11 \leq d_{i-3} = 9 + b_{i-4} - b_{i-3}.
  $$
  This holds only if $b_{i-4} = 2$ and $b_{i-3} = 0$. Thus, $d_{i-3} = 11$. Note that $s_{i-3} = d_{i-3} - 3 c_{i-3} + c_{i-4} \ge 0$. We then have $c_{i-4} \geq 3 c_{i-3} - d_{i-3} = 4$. By Lemma~\ref{lem-ci2}, $c_{i-4} = 4$. It then follows that $d_{i-4} \geq 3c_{i-4} - c_{i-5} \geq 7$. By (\ref{eqn-d2}), we get 
  \begin{eqnarray*}
    7 \leq d_{i-4} & = & 9 - 2 a_{i-5} + a_{i-4} - b_{i-6} + b_{i-5} - b_{i-4} \\ 
    & = & 7 - 2 a_{i-5} - b_{i-6} \\
    & \leq & 7.
  \end{eqnarray*}
  Then $d_{i-4} = 7$ and $a_{i-5} = b_{i-6} = 0$. By (\ref{eqn-s2}), we have $c_{i-5} \geq 3 c_{i-4} - d_{i-4} = 5$ and then $c_{i-5} = 5$. On one hand, since $s_{i-5} \geq 0$, we have $d_{i-5} \geq 3 c_{i-5} - c_{i-6} \geq 11$ by Lemma~\ref{lem-ci2}. On the other hand, (\ref{eqn-d2}) becomes
  \begin{eqnarray*}
    d_{i-5} & = & 9 - 2 a_{i-6} + a_{i-5} - b_{i-7} + b_{i-6} - b_{i-5} \\
    & = & 9 - 2a_{i-6} - b_{i-7} - b_{i-5} \\
    & \leq & 9, 
  \end{eqnarray*}
  which is a contradiction.
\qed
\end{proof}

\begin{lemma}\label{lem-ci5}
  If for some $i$, $c_i = 5$ and $c_{i-1} = \cdots = c_{i-r} = 4$ for $r \ge 1$, then $c_{i-r-1} \leq 4$.
\end{lemma}

\begin{proof}
  For $r=1$ and $r=2$, the conclusion follows from Lemma~\ref{lem-ci3} and Lemma~\ref{lem-ci4}, respectively. For $r > 2$, suppose to the contrary that $c_{i-r-1} = 5$. Since $s_i = d_i - 3 c_i + c_{i-1} \geq 0$ for each $i$, we have
  \begin{eqnarray*}
    d_i & \geq & 3 c_i - c_{i-1} = 11 ,\\
    d_{i-1} & \geq & 3 c_{i-1} - c_{i-2} = 8, \\
    & & \vdots \\
    d_{i-r} & \geq & 3 c_{i-r} - c_{i-r-1} = 7, \\
    d_{i-r-1} & \geq & 3 c_{i-r-1} - c_{i-r-2} \geq 11,
  \end{eqnarray*}
  where in the last inequality we used the fact that $c_{i-r-2} \leq 4$ by Lemma~\ref{lem-ci2}. Combining with (\ref{eqn-d2}), we get 
  \begin{eqnarray}
    11 & \leq & d_i = 9 - 2a_{i-1} + a_i - b_{i-2} + b_{i-1} - b_i, \label{eqn-ci51} \\
    8 & \leq & d_{i-1} = 9 - 2a_{i-2} + a_{i-1} - b_{i-3} + b_{i-2} - b_{i-1}, \label{eqn-ci52} \\
    & & \vdots \nonumber \\
    7 & \leq & d_{i-r} = 9 - 2a_{i-r-1} + a_{i-r} - b_{i-r-2} + b_{i-r-1} - b_{i-r}, \label{eqn-ci53} \\
    11 & \leq & d_{i-r-1} = 9 - 2a_{i-r-2} + a_{i-r-1} - b_{i-r-3} + b_{i-r-2} - b_{i-r-1}. \label{eqn-ci54}
  \end{eqnarray}
Adding up both two adjacent inequalities above, we have
\begin{eqnarray}
  19 & \leq & 18 - 2a_{i-2} - a_{i-1} + a_i - b_{i-3} - b_i, \label{eqn-ci55} \\
  16 & \leq & 18 - 2a_{i-3} - a_{i-2} + a_{i-1} - b_{i-4} - b_{i-1}, \label{eqn-ci56} \\
  & & \vdots \nonumber \\
  15 & \leq & 18 - 2a_{i-r-1} - a_{i-r} + a_{i-r+1} - b_{i-r-2} - b_{i-r+1}, \label{eqn-ci57} \\
  18 & \leq & 18 - 2a_{i-r-2} - a_{i-r-1} + a_{i-r} - b_{i-r-3} - b_{i-r}. \label{eqn-ci58}
\end{eqnarray}
Notice that (\ref{eqn-ci55}) can be satisfied only if $a_{i-2} = 0$ and $a_{i-1} \leq 1$. Similarly, we get $a_{i-3} \leq 1, \ldots, a_{i-r} \leq 1$ and $a_{i-r-2} = 0$.   
 
  Now we prove that $c_{i-r-2} = 4$. Note that $s_{i-r-1} = d_{i-r-1} - 3 c_{i-r-1} + c_{i-r-2} \geq 0$. Then we have $c_{i-r-2} \geq 3 c_{i-r-1} - d_{i-r-1} \geq 2$. By Lemma~\ref{lem-ci2}, we have $2 \leq c_{i-r-2} \leq 4$. 
  
  If $c_{i-r-2} = 2$, it follows that 
  \begin{equation}\label{eqn-ci59}
  d_{i-r-1} \geq 3 c_{i-r-1} - c_{i-r-2} = 13.
\end{equation}
Also note that $d_{i-r} \geq 3 c_{i-r} - c_{i-r-1} = 7$. Adding up the two inequalities on $d_{i-r-1}$ and $d_{i-r}$ by (\ref{eqn-d2}), we get 
  $$
  20 \leq d_{i-r} + d_{i-r-1} = 18 - 2 a_{i-r-2} - a_{i-r-1} + a_{i-r} - b_{i-r-3} - b_{i-r},
  $$
  which can be satisfied only if $a_{i-r} = 2$ and $a_{i-r-1} = a_{i-r-2} = b_{i-r} = b_{i-r-3} = 0$. Then we have 
  \begin{eqnarray*}
  d_{i-r-1} & = & 9 - 2 a_{i-r-2} + a_{i-r-1} - b_{i-r-3} + b_{i-r-2} - b_{i-r-1} \\
  & = & 9 + b_{i-r-2} - b_{i-r-1} \\
  & \leq & 11,
\end{eqnarray*}
which contradicts to (\ref{eqn-ci59}).   

If $c_{i-r-2} = 3$, we have $d_{i-r-1} \geq 3 c_{i-r-1} - c_{i-r-2} = 12$. Adding up the following two inequalities
\begin{eqnarray}
  12 & \leq & d_{i-r-1} = 9 - 2 a_{i-r-2} + a_{i-r-1} - b_{i-r-3} + b_{i-r-2} - b_{i-r-1}, \label{eqn-ci510} \\
  7 & \leq & d_{i-r} = 9 - 2 a_{i-r-1} + a_{i-r} - b_{i-r-2} + b_{i-r-1} - b_{i-r}, \nonumber 
\end{eqnarray}
we get
$$
19 \leq 18 - 2a_{i-r-2} - a_{i-r-1} + a_{i-r} - b_{i-r-3} - b_{i-r}.
$$
Since $a_{i-r} \leq 1$, this inequality holds only if $a_{i-r} = 1$ and $a_{i-r-2} = a_{i-r-1} = b_{i-r} = b_{i-r-3} = 0$. Back to (\ref{eqn-ci510}), we have 
\begin{eqnarray*}
  12 \leq  d_{i-r-1} & = & 9 + b_{i-r-2} - b_{i-r-1} \\
  & \leq & 11,
\end{eqnarray*}
which is a contradiction. 

Thus, we have $c_{i-r-2} = 4$. Note that $s_{i-r-2} = d_{i-r-2} - 3 c_{i-r-2} + c_{i-r-3} \geq 0$. Then $d_{i-r-2} \geq 3 c_{i-r-2} - c_{i-r-3} \geq 7$. By (\ref{eqn-d2}), we get
\begin{equation}\label{eqn-ci511}
  7 \leq d_{i-r-2} = 9 - 2 a_{i-r-3} + a_{i-r-2} - b_{i-r-4} + b_{i-r-3} -b_{i-r-2}.
\end{equation}
Adding $11 \leq d_{i-r-1}$, we have
\begin{equation}\label{eqn-ci512}
  18 \leq 18 - 2 a_{i-r-3} - a_{i-r-2} + a_{i-r-1} - b_{i-r-4} - b_{i-r-1}.
\end{equation}
It then follows that $a_{i-r-3} \leq 1$. Now we discuss the possible values of $a_{i-r-3}$. 
\begin{itemize}
  \item[I.] $a_{i-r-3} = 1$. 
    
    From (\ref{eqn-ci512}), it follows that $a_{i-r-1} = 2$ and $b_{i-r-1} = b_{i-r-4} = 0$. By (\ref{eqn-ci54}) and (\ref{eqn-ci511}), we have 
    \begin{eqnarray*}
      7 & \leq & d_{i-r-2} =  7 + b_{i-r-3} - b_{i-r-2},\\
      11 & \leq & d_{i-r-1} = 11 - b_{i-r-3} + b_{i-r-2}.
    \end{eqnarray*}
    Thus, $b_{i-r-3} = b_{i-r-2}$, $d_{i-r-1} = 11$ and $d_{i-r-2} = 7$. Since $s_{i-r-2} = d_{i-r-2} - 3 c_{i-r-2} + c_{i-r-3} \geq 0$, we have $c_{i-r-3} \geq 3 c_{i-r-2} - d_{i-r-2} = 5$ and then $c_{i-r-3} = 5$. It then follows that $d_{i-r-3} \geq 3 c_{i-r-3} - c_{i-r-4} \geq 11$ because $c_{i-r-4} \leq 4$ by Lemma~\ref{lem-ci2}. However, by (\ref{eqn-d2}), we have 
    \begin{eqnarray*}
      11 \leq d_{i-r-3} & = & 9 - 2 a_{i-r-4} + a_{i-r-3} - b_{i-r-5} + b_{i-r-4} - b_{i-r-3} \\
      & = & 10 - 2 a_{i-r-4} - b_{i-r-5} - b_{i-r-3} \\
      & \leq & 10,
    \end{eqnarray*}
    leading to a contradiction.
  
  \item[II.] $a_{i-r-3} = 0$. 
    
    Then inequality (\ref{eqn-ci512}) becomes $18 \leq 18 + a_{i-r-1} - b_{i-r-4} - b_{i-r-1}$. We continue to discuss the possible values of $a_{i-r-1}$.
    \begin{itemize}
      \item[II.(i).] $a_{i-r-1} = 0$. 
	
	By (\ref{eqn-ci512}), we have $a_{i-r} = b_{i-4-1} = b_{i-r-4} = 0$. Then inequality (\ref{eqn-ci54}) becomes $11 \leq d_{i-r-1} = 9 - b_{i-r-3} + b_{i-r-2}$, which implies that $b_{i-r-2}= 2$ and $b_{i-r-3} = 0$. It then follows from (\ref{eqn-ci511}) that $d_{i-r-2} = 7$. Since $s_{i-r-2} = d_{i-r-2} - 3 c_{i-r-2} + c_{i-r-3} \geq 0$, we have $c_{i-r-3} \geq 3 c_{i-r-2} - d_{i-r-2} = 5$ and then $c_{i-r-3} = 5$. There is a contradiction as
	the following.
	\begin{eqnarray*}
	  11 & \leq & 3 c_{i-r-3} - c_{i-r-4} \\
	  & \leq & d_{i-r-3} \\
	  & = & 9 - 2a_{i-r-4} + a_{i-r-3} - b_{i-r-5} + b_{i-r-4} - b_{i-r-3} \\
	  & = & 9 - 2 a_{i-r-4} - b_{i-r-5} - b_{i-r-3}\\
	  & \leq & 9.
	\end{eqnarray*}

      \item[II.(ii).] $a_{i-r-1} = 1$. 
	
	By (\ref{eqn-ci58}), we get $18 \leq 17 + a_{i-r} - b_{i-r-3} - b_{i-r}$. Together with $a_{i-r} \leq 1$, we have $a_{i-r} = 1$ and $b_{i-r} = b_{i-r-3} = 0$. From (\ref{eqn-ci53}) and (\ref{eqn-ci54}), it follows that 
	\begin{eqnarray*}
	  7 & \leq & 8 - b_{i-r-2} + b_{i-r-1},\\
	  11 & \leq & 10 + b_{i-r-2} - b_{i-r-1}.
	\end{eqnarray*}
	Thus, $b_{i-r-2} - b_{i-r-1} = 1$. There are the following two subcases. 
	\begin{itemize}
	  \item[II.(ii).a.] $b_{i-r-2} = 2$ and $b_{i-r-1} = 1$. 
	    
	    By (\ref{eqn-ci511}), we have 
	    $$
	    7 \leq d_{i-r-2} = 7 - b_{i-r-4},
	    $$
	    which implies that $b_{i-r-4} = 0$ and $d_{i-r-2} = 7$. With similar argument to Case II.(i), there is also a contradiction on $d_{i-r-3}$.

	  \item[II.(ii).b.] $b_{i-r-2} = 1$ and $b_{i-r-1} = 0$. 
	    
	    Then (\ref{eqn-ci511}) becomes $7 \leq 8 - b_{i-r-4}$, which implies that $b_{i-r-4} \leq 1$. 
	    
	    If $b_{i-r-4} = 1$, we have $d_{i-r-2} = 7$. It then follows that $c_{i-r-3} \geq 3c_{i-r-2} - d_{i-r-2} = 5$. Similar argument to Case II.(i) also leads to a contradiction on $d_{i-r-3}$. 
	    
	    If $b_{i-r-4} = 0$, we summarize the information on $\{a_i\}$ and $\{b_i\}$ in the tables $I_a$ and $I_b$, respectively.
	    \begin{equation*}
	      I_a := \left[ \begin{array}{cccc}
		a_{i-r} & a_{i-r-1} & a_{i-r-2} & a_{i-r-3} \\
		1 & 1 & 0 & 0
	      \end{array}\right]
	    \end{equation*}
	    \begin{equation*}
	      I_b := \left[ \begin{array}{ccccc}
		b_{i-r} & b_{i-r-1} & b_{i-r-2} & b_{i-r-3} & b_{i-r-4} \\
		0 & 0 & 1 & 0 & 0 
	      \end{array}\right]
	    \end{equation*}
	    Recall that
	    $$
	    8 \leq d_{i-j} = 9 - 2 a_{i-j-1} + a_{i-j} - b_{i-j-2} + b_{i-j-1} - b_{i-j},
	    $$
	    for $1 \leq j \leq r -1$ and
	    $$
	    a_{i-j} \leq 1,
	    $$
	    for $3 \leq j \leq r$. For $j = r-1$, we have 
	    $$
	    8 \leq d_{i-r+1} = 9 - 2a_{i-r} + a_{i-r+1} - b_{i-r-1} + b_{i-r} - b_{i-r+1}
	    $$
	    and $a_{i-r+1} \leq 1$. It then follows from $I_a$ and $I_b$ that $a_{i-r+1} = 1$ and $b_{i-r+1}=0$. Similarly, if we continue to do this for $j = r-2, \ldots, 2$, we get $a_{i-2} = 1$ and $b_{i-1} = 0$, which contradicts to the fact that $a_{i-2} = 0$ by (\ref{eqn-ci55}). 

	\end{itemize}

      \item[II.(iii).] $a_{i-r-1} = 2$. 
	
	By (\ref{eqn-ci58}), we have
	\begin{eqnarray*}
	18 & \leq & 18 - 2a_{i-r-2} - a_{i-r-1} + a_{i-r} - b_{i-r-3} - b_{i-r} \\
	& = & 16 + a_{i-r} - b_{i-r-3} - b_{i-r}.
      \end{eqnarray*}
      It then follows that $a_{i-r} = 2$, which is a contradiction to $a_{i-r} \leq 1$. 

  \end{itemize}

\end{itemize}
To sum up, if $c_{i-r-1} = 5$, we can always derive a contradiction. Thus, the conclusion follows. 
\qed
\end{proof}

\begin{corollary}\label{coro-ci6}
  Let $t$ be an integer with $1 \leq t \leq m$. If $c_i = c_{i-t} = 5$ for some $i$, and $c_{i-1} \leq 4$, $c_{i-2} \leq 4$, $\ldots$, $c_{i-(t-1)} \leq 4$, then there exists some $\ell$ with $1 \leq \ell \le t-1$ such that $c_{i-\ell} < 4$.
\end{corollary}

\begin{proof}
  Assume to the contrary that for every $\ell \in \{1, 2, \ldots, t-1\}$, we have $c_{i-\ell} = 4$. That is,
  $$
  c_i = 5, \quad c_{i-1} = c_{i-2} = \cdots = c_{i-(t-1)} = 4.
  $$
  Then by Lemma~\ref{lem-ci5}, we have $c_{i-t} \leq 4$, contradicting to the condition that $c_{i-t} = 5$, which completes the proof.
  \qed
\end{proof}

\begin{proof}[Theorem~\ref{thm-goal42}]
  As discussed before, it suffices to prove (\ref{eqn-ci0}). If $c_i \leq 4$ for all $0 \leq i \leq m-1$, inequality (\ref{eqn-ci0}) directly holds. We now assume that there exists an $h$ with $0 \leq h \leq m-1$, such that $c_h = 5$. It is easily seen that $c_h, c_{h-1}, \ldots, c_{h-(m-1)}$ is a shift reverse of $c_0, c_1, \ldots, c_{m-1}$. Without loss of generality, we further assume that $c_{h-i_1} = c_{h-i_2} = \cdots = c_{h-i_s} = 5$, where $0 = i_1 \leq i_2 \leq
  \cdots \leq i_s < m$, $s \geq 1$ and $c_{h-j} \leq 4$ for each $j \in \{0, 1, \ldots, m-1 \}\setminus \{i_1, i_2, \ldots, i_s\}$. By Lemma~\ref{lem-ci2}, we have $i_2 \geq i_1 + 2$, $i_3 \geq i_2 + 2$, $\ldots$, $i_s \geq i_{s-1} + 2$ and $m \geq i_s + 2$. Using Corollary 
  ~\ref{coro-ci6}, we have the following inequalities on the sum of $c_i$'s in each segment.
  \begin{eqnarray*}
    c_{h-i_1} + c_{h-i_1 -1} + \cdots + c_{h-(i_2-1)} & \leq & 4 (i_2 - i_1) \\
    c_{h-i_2} + c_{h-i_2- 1} + \cdots + c_{h-(i_3 - 1)} & \leq & 4 (i_3 - i_2) \\
    \vdots & & \\
    c_{h-i_s} + c_{h-i_s - 1} + \cdots + c_{h-(m-1)} & \leq & 4 (m - i_s)
  \end{eqnarray*}
  Summing up all the inequalities above, we get
  $$
  \sum_{i=0}^{m-1} c_i = \sum_{j=0}^{m-1} c_{h-j} \leq 4m,
  $$
which completes the proof.
\qed
\end{proof}



\end{document}